\documentclass[dvipdfmx]{article}

\usepackage{amsmath,amssymb,amsthm,mathrsfs}
\usepackage[dvipdfmx]{graphicx}
\usepackage{here}
\newtheorem{theo}{Theorem}
\newtheorem{lemm}[theo]{Lemma}
\newtheorem{prop}[theo]{Proposition}
\newtheorem{cor}[theo]{Corollary}

\title{On special values at integers of $L$-functions of Jacobi theta products of weight $3$}
\author{Ryojun Ito\thanks{
Department of Mathematics and Informatics, Graduate School of Science, Chiba University, 
Yayoicho 1-33, Inage, Chiba, 263-8522 Japan. E-mail: afua9032@chiba-u.jp,  
2010 Mathematics Subject Classification: 11F27, 11F67, 33C20, 33C99       
keywords: theta series, $L$-values for theta products, 
generalized hypergeometric functions, Kamp{\'e} de F{\'e}riet hypergeometric functions. }}

\date{}

\begin{document}

\maketitle

\begin{abstract}
In this paper, we consider $L$-functions of modular forms of weight 3, which are products of the Jacobi theta series, and 
express their special values at $s=3$, $4$ in terms of special values of Kamp{\'e} de F{\'e}riet hypergeometric functions. 
Moreover, via $L$-values, we give some relations between special values of Kamp{\'e} de F{\'e}riet hypergeometric functions and 
generalized hypergeometric functions.  
\end{abstract}

\section{Introduction and Main Results}
Let $f$ be a modular form of weight $k$ with $q$-expansion 
$f(q) =  \sum_{n=0}^{\infty} a_{n}q^{n}$ ($q = e^{2\pi i \tau}$, ${\rm Im}(\tau)>0$).
Then its $L$-function $L(f, s) = \sum_{n=1}^{\infty} a_{n} / n^{s}$ converges absolutely on ${\rm Re}(s) > k+1$ 
(${\rm Re}(s) > k/2+1$ if $f$ is a cusp form). 
When the Fricke involution image $f^{\#}$ of $f$ is also a modular form, then $L(f,s)$ is meromorphically continued to 
the whole complex plane with a possible simple pole at $s=k$, and is entire when $f^{\#}(0)=0$ (see \cite[Theorem 3.2]{shimura}). 
In this paper, we consider the case when $f(q)$ is a product of the Jacobi theta series
\begin{align*}
\theta_{2}(q) := \sum_{n\in\mathbb{Z}}q^{\left( n + \frac{1}{2} \right)^{2}}, \hspace{2mm}
\theta_{3}(q) := \sum_{n \in\mathbb{Z}} q^{n^{2}}, \hspace{2mm}
\theta_{4}(q) := \sum_{n \in\mathbb{Z}} (-1)^{n}q^{n^{2}},
\end{align*}
which are modular forms of weight $1/2$, and satisfies the condition $f^{\#}(0)=0$.

In \cite{rz}, by an analytic method, Rogers and Zudilin expressed $L(f, 2)$ for some theta products $f(q)$ of weight 2
in terms of special values of generalized hypergeometric functions 
\begin{align*}
{_{p+1}F_{p}}\left[ 
\left.
\begin{matrix}
a_{1},a_{2}, \dots , a_{p+1} \\
b_{1}, b_{2}, \dots , b_{p} 
\end{matrix}
\right| z
\right]
:=
\sum_{n=0}^{\infty} \frac{(a_{1})_{n} \cdots (a_{p+1})_{n}}{(b_{1})_{n} \cdots (b_{p})_{n}} \frac{z^{n}}{(1)_{n}},
\end{align*}
where $(a)_{n} := \Gamma(a+n) / \Gamma(a)$ denotes the Pochhammer symbol.
Other known results of hypergeometric expressions of $L$-values are the following. 
\begin{enumerate}
\item Otsubo \cite{otsubo} expressed $L(f, 2)$ for some theta products $f(q)$ of weight $2$ 
in terms of ${_{3}F_{2}}(1)$ via regulators.

\item Rogers \cite{rog}, Rogers-Zudilin \cite{rz}, Zudilin \cite{zud} and 
the author \cite{ito1} expressed $L(f, 2)$ for some theta products $f(q)$ of weight $2$ in terms of ${_{3}F_{2}}(1)$ by using
the Rogers-Zudilin method. Furthermore, Zudilin \cite{zud} expressed $L(f, 3)$ for the theta product which corresponds 
to the elliptic curve of conductor $32$ in terms of ${_{4}F_{3}}(1)$.

\item Rogers-Wan-Zucker \cite{rwz} expressed $L(f, 2)$ (resp. $L(f, 3)$, $L(f,4)$) for some quotients $f(q)$ of 
the Dedekind eta function $\eta(q) := q^{1/24}\prod_{n=1}^{\infty} (1-q^{n})$ of weight $3$ (resp. $4$, $5$) 
in terms of special values of generalized hypergeometric functions or the gamma function by an analytic method. 
The author \cite{ito2} expressed $L(f, 1)$ (hence the values at $2$ by the functional equations) for some theta products $f(q)$ 
of weight $3$ in terms of ${_{3}F_{2}}(1)$ by using the Rogers-Zudilin method. 

\item Samart \cite{samart} expressed $L(f, 3)$ for some eta quotients $f(q)$ of weight $3$ in terms of ${_{5}F_{4}}(1)$ 
via Mahler measures. 

\end{enumerate}

In this paper, we consider the following normalized Jacobi theta products of weight $3$
\begin{align*}
f(q) = \frac{1}{16}\theta_{2}^{4}(q)\theta_{4}^{2}(q), \hspace{2mm} g(q) = \frac{1}{16} \theta_{2}^{4}(q)\theta_{4}^{2}(q^{2}).
\end{align*}
We remark that $f(q)$ is an Eisenstein series twisted by some Dirichlet characters and 
$g(q)$ is the cusp form corresponding to the Kummer K3 surface 
defined by $z^{2} = x(x^{2}-1)y(y^{2}-1)$ (cf. \cite[Theorem 7.4]{topyui}).

The aim of this paper is to express $L(f, n)$ and $L(g, n)$ for $n=3,4$ 
in terms of special values of the Kamp{\'e} de F{\'e}riet hypergeometric function \cite{appkampe, sk}
\begin{align*}
&F_{C;D;D^{\prime}}^{A;B;B^{\prime}}\left[\left.
\begin{matrix}
a_{1}, \dots, a_{A}  \\
c_{1}, \dots, c_{C}
\end{matrix}
; 
\begin{matrix}
b_{1} , \dots, b_{B} \\
d_{1}, \dots, d_{D}
\end{matrix}
;
\begin{matrix}
b_{1}^{\prime}, \dots, b_{B^{\prime}}^{\prime} \\
d_{1}^{\prime}, \dots, d_{D^{\prime}}^{\prime}
\end{matrix}
\right| x, y
\right] \\
&:= 
\sum_{m,n=0}^{\infty} 
\frac{
\prod_{i=1}^{A} (a_{i})_{m+n} \prod_{i=1}^{B} (b_{i})_{m} \prod_{i=1}^{B^{\prime}} (b_{i}^{\prime})_{n} 
}{\prod_{i=1}^{C} (c_{i})_{m+n} \prod_{i=1}^{D} (d_{i})_{m} \prod_{i=1}^{D^{\prime}} (d_{i}^{\prime})_{n}
} \frac{x^{m}y^{n}}{(1)_{m}(1)_{n}},
\end{align*}
which is a two-variable generalization of generalized hypergeometric functions.

The main results are the following. 

\begin{theo}\label{val3}

\begin{enumerate}
\item 
\begin{align*}
L(f, 3) = \frac{\pi^{2}}{96} 
F_{1;1;1}^{1;2;2}\left[\left.
\begin{matrix}
2 \\
\frac{5}{2}
\end{matrix}
; 
\begin{matrix}
1, 1 \\
2
\end{matrix}
;
\begin{matrix}
\frac{1}{2}, \frac{1}{2} \\
1
\end{matrix}
\right| 1, 1
\right]. 
\end{align*}

\item 
\begin{align*}
L(g, 3) = \frac{\pi^{3}}{128}
F_{1;1;1}^{1;2;2} \left[ \left. 
\begin{matrix}
\frac{3}{2} \\
2
\end{matrix}
;
\begin{matrix}
\frac{1}{2}, 1 \\
\frac{3}{2}
\end{matrix}
;
\begin{matrix}
\frac{1}{2}, \frac{1}{2} \\
1
\end{matrix}
\right|
1,1
\right] .
\end{align*}
\end{enumerate}
\end{theo}

\begin{theo}\label{val4}
\begin{enumerate}
\item 
\begin{align*}
 L(f,4)  
=
\frac{\pi^{3}}{288} \left(
3  F_{1;2;1}^{1;3;2}\left[ \left.
\begin{matrix}
\frac{1}{2} \\
\frac{3}{2}
\end{matrix}
;
\begin{matrix}
1, 1, 1 \\
\frac{3}{2}, \frac{3}{2}
\end{matrix}
; 
\begin{matrix}
\frac{1}{2}, \frac{1}{2} \\
1
\end{matrix}
\right| 1, 1
\right]
+ 
 F_{1;2;1}^{1;3;2}\left[\left.
\begin{matrix}
\frac{3}{2} \\
\frac{5}{2}
\end{matrix}
;
\begin{matrix}
1, 1, 1 \\
\frac{3}{2}, \frac{3}{2}
\end{matrix}
; 
\begin{matrix}
\frac{1}{2}, \frac{1}{2} \\
1
\end{matrix}
\right| 1, 1
\right] 
\right) . 
\end{align*}

\item  
\begin{align*}
L(g, 4) = \frac{\pi^{4}}{768} \left( 
2 F_{1;2;1}^{1;3;2}\left[ \left.
\begin{matrix}
\frac{1}{2} \\
1
\end{matrix}
;
\begin{matrix}
1, 1, 1 \\
\frac{3}{2}, \frac{3}{2}
\end{matrix}
;
\begin{matrix}
\frac{1}{2}, \frac{1}{2} \\
1
\end{matrix}
\right| 1,1
\right] 
+ F_{1;2;1}^{1;3;2}\left[ \left.
\begin{matrix}
\frac{1}{2} \\
2
\end{matrix}
;
\begin{matrix}
1, 1, 1 \\
\frac{3}{2}, \frac{3}{2}
\end{matrix}
;
\begin{matrix}
\frac{1}{2}, \frac{1}{2} \\
1
\end{matrix}
\right| 1,1
\right] \right) . 
\end{align*}
\end{enumerate}
\end{theo}
We remark that the double series $F_{A:B:C}^{A:B+1:C+1}(x,y)$ converges absolutely 
on $|x| \leqq 1$ and $|y| \leqq 1$ when the parameters satisfy the following conditions \cite[Theorem 1]{hms}
\begin{align*}
&{\rm Re}\left( \sum_{i=1}^{A} c_{i} + \sum_{i=1}^{B} d_{i} - \sum_{i=1}^{A} a_{i} - \sum_{i=1}^{B+1} b_{i} \right) > 0, \\
&{\rm Re}\left( \sum_{i=1}^{A} c_{i} + \sum_{i=1}^{C} d^{\prime}_{i} - \sum_{i=1}^{A} a_{i} - \sum_{i=1}^{C+1} b_{i}^{\prime} \right) > 0, \\
&{\rm Re}\left( \sum_{i=1}^{A} c_{i} + \sum_{i=1}^{B} d_{i} + \sum_{i=1}^{C} d_{i}^{\prime} 
- \sum_{i=1}^{A} a_{i} - \sum_{i=1}^{B+1} b_{i} -\sum_{i=1}^{C+1} b_{i}^{\prime} \right) > 0.
\end{align*}

We prove Theorems \ref{val3} and \ref{val4} by using the Rogers-Zudilin method. The strategy of the method is as follows. 
For $h(q) = \sum_{n=0}^{\infty} a_{n}q^{n}$ and $n \in \mathbb{Z}_{\geqq 1}$, the value $L(h, n)$ is obtained by 
the Mellin transformation of $h(q)$
\begin{align}
L(h, n) = \frac{(-1)^{n-1}}{\Gamma (n)} \int_{0}^{1} (h(q) -a_{0}) (\log q)^{n-1} \frac{dq}{q}. \label{mellin}
\end{align}
The key to express the value $L(h,n)$ in terms of special values of hypergeometric functions is the following transformation formulas 
\begin{align}
\theta_{3}^{2}(q) 
= {_{2}F_{1}}\left[ \left.
\begin{matrix}
\frac{1}{2}, \frac{1}{2} \\
1 
\end{matrix} 
\right| \alpha 
\right], \hspace{5mm} 
\theta_{3}^{4}(q)\frac{dq}{q} = \frac{d\alpha}{\alpha (1-\alpha)}, \label{trans}
\end{align}
where $\alpha := \theta_{2}^{4}(q) / \theta_{3}^{4}(q)$. 
The former is \cite[p.101, Entry 6]{ramanujan3}, and the latter follows from the former and \cite[p.87, Entry 30]{ramanujan2}.
By these formulas, we reduce the integral \eqref{mellin} to 
an integral of the form
\begin{align*}
\int_{0}^{1} P(\alpha)
{_{p+1}F_{p}}\left[ \left.
\begin{matrix}
a_{1}, a_{2}, \dots , a_{p+1} \\
b_{1}, \dots, b_{p} 
\end{matrix} 
\right| \alpha 
\right]
{_{2}F_{1}}\left[ \left.
\begin{matrix}
\frac{1}{2}, \frac{1}{2} \\
1 
\end{matrix} 
\right| \alpha 
\right]
\frac{d\alpha}{\alpha (1-\alpha)} , 
\end{align*}
where $P(\alpha )$ is a polynomial in $\alpha^{k}(1-\alpha)^{l}$ for various $k$ and $l$.  
Then, by simple computations, we obtain hypergeometric expressions of $L$-values.

Finally, we remark that we have simpler hypergeometric expressions of the $L$-values $L(f, 3)$, $L(f, 4)$ and $L(g,3)$.   
Since $f(q)$ is the Eisenstein series twisted by the Dirichlet characters $\chi_{-4}(n) := {\rm Im}(i^{n})$ 
and $\psi (n) := (-1)^{n-1}$ \cite[p. 281, Lemma 3.32 (3.85)]{cooper}  
\begin{align*}
f(q) = \frac{1}{16}\theta_{2}^{4}(q)\theta_{4}^{2}(q)
= \sum_{n=1}^{\infty} \frac{(-1)^{n+1}n^{2}q^{n}}{1+q^{2n}} = \sum_{n, k=1}^{\infty} \psi (n) n^{2} \chi_{-4}(k)q^{nk},   
\end{align*}
we have  
\begin{align*}
L(f, s) = \sum_{n,k=1}^{\infty} \frac{\psi(n) n^{2}\chi_{-4}(k)}{n^{s}k^{s}}  
= L(\psi, s-2) L(\chi_{-4}, s),
\end{align*}
where $L(\chi, s)$ is the Dirichlet $L$-function associated to a Dirichlet character $\chi$. 
We know $L(\psi , 1) = \log 2$ and $L(\chi_{-4}, 3) = \pi^{3}/32$, hence we obtain 
\begin{align*}
L(f, 3) = \frac{\pi^{3} \log 2}{32}.
\end{align*}
Moreover, the value $L(f, 4)$ can be expressed in terms of ${_{5}F_{4}}$, since we have 
$L(\psi, 2) = \pi^{2}/12$ and 
\begin{align*}
L(\chi_{-4}, 4) = 
{_{5}F_{4}}\left[ \left. 
\begin{matrix}
\frac{1}{2}, \dots, \frac{1}{2}, 1 \\
\frac{3}{2}, \dots , \frac{3}{2}
\end{matrix}
\right| -1 \right]
= {_{5}F_{4}}\left[ \left. 
\begin{matrix}
\frac{1}{4}, \dots, \frac{1}{4} , 1 \\
\frac{5}{4}, \dots , \frac{5}{4}
\end{matrix}
\right| 1 \right]
- \frac{1}{81}
{_{5}F_{4}}\left[ \left. 
\begin{matrix}
\frac{3}{4}, \dots, \frac{3}{4}, 1 \\
\frac{7}{4}, \dots , \frac{7}{4}
\end{matrix}
\right| 1 \right], 
\end{align*}
which easily follows from
\begin{align*}
2n+1 = \frac{\left( \frac{3}{2} \right)_{n}}{\left( \frac{1}{2} \right)_{n}}, \hspace{3mm}
4n+1 = \frac{\left( \frac{5}{4} \right)_{n}}{\left( \frac{1}{4} \right)_{n}}, \hspace{3mm}
4n+3 = \frac{3\left( \frac{7}{4} \right)_{n}}{\left( \frac{3}{4} \right)_{n}} .
\end{align*}
Similarly, we have a simpler expression of $L(g,3)$ \cite[Corollary 1.3]{samart}
\begin{align*}
L(g, 3) = \frac{\pi^{3}}{1024} \left( 
48 \log 2
- 
{_{5}F_{4}}\left. 
\left[
\begin{matrix}
\frac{3}{2}, \frac{3}{2}, \frac{3}{2}, 1, 1 \\
2, 2, ,2 ,2
\end{matrix}
\right| 1 
\right] 
\right) .
\end{align*}
Therefore we obtain via $L$-values the following reduction formulas for Kamp{\'e} de F{\'e}riet hypergeometric functions.

\begin{cor}
\begin{enumerate}
\item

\begin{align*}
F_{1;1;1}^{1;2;2}\left[\left.
\begin{matrix}
2 \\
\frac{5}{2}
\end{matrix}
; 
\begin{matrix}
1, 1 \\
2
\end{matrix}
;
\begin{matrix}
\frac{1}{2}, \frac{1}{2} \\
1
\end{matrix}
\right| 1, 1
\right] = 3\pi\log 2. 
\end{align*}

\item
\begin{align*}
8F_{1;1;1}^{1;2;2} \left[ \left. 
\begin{matrix}
\frac{3}{2} \\
2
\end{matrix}
;
\begin{matrix}
\frac{1}{2}, 1 \\
\frac{3}{2}
\end{matrix}
;
\begin{matrix}
\frac{1}{2}, \frac{1}{2} \\
1
\end{matrix}
\right|
1,1
\right] 
=48 \log 2 - 
{_{5}F_{4}}
\left[ \left. 
\begin{matrix}
\frac{3}{2}, \frac{3}{2}, \frac{3}{2}, 1, 1 \\
2, 2, ,2 ,2
\end{matrix}
\right| 1 
\right] . 
\end{align*}

\item
\begin{align*}
&\frac{\pi}{24} \left(
3  F_{1;2;1}^{1;3;2}\left[ \left.
\begin{matrix}
\frac{1}{2} \\
\frac{3}{2}
\end{matrix}
;
\begin{matrix}
1, 1, 1 \\
\frac{3}{2}, \frac{3}{2}
\end{matrix}
; 
\begin{matrix}
\frac{1}{2}, \frac{1}{2} \\
1
\end{matrix}
\right| 1, 1
\right]
+ 
 F_{1;2;1}^{1;3;2}\left[\left.
\begin{matrix}
\frac{3}{2} \\
\frac{5}{2}
\end{matrix}
;
\begin{matrix}
1, 1, 1 \\
\frac{3}{2}, \frac{3}{2}
\end{matrix}
; 
\begin{matrix}
\frac{1}{2}, \frac{1}{2} \\
1
\end{matrix}
\right| 1, 1
\right] 
\right)  \\
&= 
{_{5}F_{4}}\left[ \left. 
\begin{matrix}
\frac{1}{2}, \dots, \frac{1}{2}, 1 \\
\frac{3}{2}, \dots , \frac{3}{2}
\end{matrix}
\right| -1 \right] . 
\end{align*}
\end{enumerate}
\end{cor}
The author does not know how to derive these formulas directly. 
It is new that the value $L(g,4)$ is expressed in terms of special values of hypergeometric functions. 
Similarly to the results above, 
one might be able to express the value $L(g,4)$ in terms of special values of generalized hypergeometric functions.

\section{Proof of Theorem \ref{val3}}\label{proof1}

We first show the following integral expressions of the $L$-values $L(f,3)$ and $L(g,3)$. 
\begin{prop}\label{intrep3}
\begin{enumerate}
\item 
\begin{align*}
L(f, 3) = \frac{\pi^{2}}{8} 
\int_{0}^{1} \theta_{2}^{4}(q) \theta_{4}^{2}(q) \sum_{r,k=1}^{\infty} \frac{q^{(2r-1)(2k-1)}}{2r-1} \frac{dq}{q}. 
\end{align*}

\item  
\begin{align*}
L(g,3) = \frac{\pi^{2}}{16}
\int_{0}^{1}
\theta_{2}^{4}(q)\theta_{4}^{2}(q)
\sum_{r,k=1}^{\infty} \frac{q^{2(r-1/2)(k-1/2)}}{2r-1} \frac{dq}{q}.
\end{align*}

\end{enumerate}
\end{prop}
\begin{proof}
We prove the formula for $L(f, 3)$ only. Similar computation leads to the remaining formula.

By \eqref{mellin}, we have
\begin{align*}
L(f, 3) = \frac{1}{2}\int_{0}^{1} \frac{1}{16} \theta_{2}^{4}(q)\theta_{4}^{2}(q) (\log q)^{2} \frac{dq}{q}.  
\end{align*}
By changing the variable $q =e^{-\pi u}$, we have 
\begin{align*}
L(f, 3) = \frac{\pi^{3}}{32} \int_{0}^{\infty} \theta_{2}^{4}(e^{-\pi u})\theta_{4}^{2}(e^{-\pi u}) u^{2}du. 
\end{align*}
If we use the involution formula for $\theta_{4}(q)$ \cite[p.40, (2.3.3)]{borweins},
the Lambert series expansions of $\theta_{2}^{2}(q)$ and $\theta_{2}^{4}(q)$ 
\cite[p.177, Theorem 3.10 (3.15) and p.196, Theorem 3.26 (3.69)]{cooper}
\begin{align}
\sqrt{u}\theta_{4}(e^{-\pi u}) &= \theta_{2}(e^{-\frac{\pi}{u}}), \label{inv} \\
\theta_{2}^{2}(q) &= 4 \sum_{n, k=1}^{\infty} \chi_{-4}(n)q^{n(k-1/2)},\hspace{4mm} \chi_{-4}(n) := {\rm Im}(i^{n}), \label{lam1}  \\
\theta_{2}^{4}(q) &= 16 \sum_{r, s=1}^{\infty} (2r-1)q^{(2r-1)(2s-1)}, \label{lam2}
\end{align}
and the substitution $u \mapsto nu / (2r-1)$, then we obtain 
\begin{align*}
L(f, 3)= 2\pi^{3}\int_{0}^{\infty} 
\left( \sum_{n,s=1}^{\infty} \chi_{-4}(n) n^{2}e^{-\pi u n(2s-1)} \right) 
\left( \sum_{r, k= 1}^{\infty} \frac{e^{- \frac{\pi  (2r-1) (k-1/2)}{u}}}{2r-1} \right)  udu. 
\end{align*}
We know that the first series in the integral above is the theta product \cite[Lemma 3.32 (3.84)]{cooper}
\begin{align*}
\sum_{n, s=1}^{\infty} \chi_{-4}(n)n^{2} q^{n (2s-1)} = \frac{1}{4} \theta_{2}^{2}(q^{2})\theta_{4}^{4}(q^{2}).
\end{align*}
By this identity and \eqref{inv}, we have 
\begin{align*}
L(f, 3) &= \frac{\pi^{3}}{2}\int_{0}^{\infty} \theta_{2}^{2}(e^{-2\pi u})\theta_{4}^{4}(e^{-2\pi u}) 
\left( \sum_{r, k= 1}^{\infty} \frac{e^{- \frac{\pi  (2r-1) (k-1/2)}{u}}}{2r-1} \right)  udu \\
&=\frac{\pi^{3}}{16}\int_{0}^{\infty} \theta_{4}^{2}(e^{-\pi/ 2 u})\theta_{2}^{4}(e^{- \pi / 2 u}) 
\left( \sum_{r, k= 1}^{\infty} \frac{e^{- \frac{\pi  (2r-1) (k-1/2)}{u}}}{2r-1}  \right) \frac{du}{u^{2}}.
\end{align*}
If we use the substitutions $u \mapsto 1/u$, $q= e^{-\pi u}$ and $q \mapsto q^{2}$, then we obtain the formula.

\end{proof}

The series in the integrals in Proposition \ref{intrep3} are hypergeometric functions. 
\begin{lemm}\label{serieswt0}
\begin{enumerate}
\item
\begin{align*}
\sum_{r,k=1}^{\infty} \frac{q^{(2r-1)(2k-1)}}{2r-1} 
= \frac{\alpha}{16} {_{2}F_{1}}
\left[ \left.
\begin{matrix}
1, 1 \\
2
\end{matrix}
\right| \alpha
\right].
\end{align*}

\item 
\begin{align*}
\sum_{r,k=1}^{\infty} \frac{q^{2(r-1/2)(k-1/2)}}{2r-1} 
= \frac{\alpha^{1/2}}{4}
{_{2}F_{1}}\left[ \left.
\begin{matrix}
\frac{1}{2}, 1 \\
\frac{3}{2}
\end{matrix}
\right| \alpha \right] .
\end{align*}
\end{enumerate}
\end{lemm}
\begin{proof}
We prove these hypergeometric expressions by using the transformation formulas \eqref{trans}. 

By \eqref{lam2}, we have
\begin{align*}
\sum_{r,k=1}^{\infty} \frac{q^{(2r-1)(2k-1)}}{2r-1} 
=\int_{0}^{q} \sum_{r,k=1}^{\infty} (2k-1)q^{(2r-1)(2k-1)} \frac{dq}{q}
=\frac{1}{16}\int_{0}^{q} \theta_{2}^{4}(q) \frac{dq}{q}.
\end{align*}
If we use \eqref{trans}, then the integral above is equal to 
\begin{align*}
\frac{1}{16}\int_{0}^{\alpha} \alpha \frac{d\alpha}{\alpha (1-\alpha)}
=-\frac{1}{16} \log (1-\alpha ) . 
\end{align*}
Since we know 
\begin{align*}
-\log (1-\alpha) = \alpha {_{2}F_{1}}
\left[ \left.
\begin{matrix}
1, 1 \\
2
\end{matrix}
\right| \alpha
\right], 
\end{align*}
we obtain the first formula. 

Similarly, we have
\begin{align*}
\sum_{r,k=1}^{\infty} \frac{q^{2(r-1/2)(k-1/2)}}{2r-1} 
=\frac{1}{32}\int_{0}^{q} \theta_{2}^{4}(q^{\frac{1}{2}}) \frac{dq}{q}
=\frac{1}{8}\int_{0}^{q} \theta_{2}^{2}(q)\theta_{3}^{2}(q) \frac{dq}{q}.
\end{align*}
Here we used $2\theta_{2}(q^{2})\theta_{3}(q^{2}) = \theta_{2}^{2}(q)$ \cite[p. 40, Entry 25 (iv)]{ramanujan3} for the last equality. 
Then, by \eqref{trans}, we obtain 
\begin{align*}
\int_{0}^{q} \theta_{2}^{2}(q)\theta_{3}^{2}(q) \frac{dq}{q}
=\int_{0}^{\alpha} \alpha^{\frac{1}{2}} \frac{d\alpha}{\alpha (1-\alpha)}. 
\end{align*}
By the integral representation of hypergeometric functions \cite[(1.6.6)]{slater}
\begin{align*}
{_{2}F_{1}} \left[ \left. 
\begin{matrix}
a,b \\
c
\end{matrix}
\right| 
z \right]
=\frac{\Gamma (c)}{\Gamma (b) \Gamma (c-b)} \int_{0}^{1} t^{b} (1-t)^{c-b} (1-zt)^{-a} \frac{dt}{t(1-t)},
\end{align*}
we have 
\begin{align*}
\int_{0}^{\alpha} t^{\frac{1}{2}} \frac{dt}{t (1-t)}
&=\alpha^{\frac{1}{2}} \int_{0}^{1} t^{\frac{1}{2}} (1-t) (1-\alpha t)^{-1} \frac{dt}{t(1-t)} \\
&=\alpha^{\frac{1}{2}} \frac{\Gamma\left( \frac{1}{2} \right) \Gamma (1)}{\Gamma\left( \frac{3}{2} \right) }
{_{2}F_{1}} \left[ \left. 
\begin{matrix}
\frac{1}{2}, 1 \\
\frac{3}{2}
\end{matrix}
\right| 
\alpha \right],
\end{align*}
hence we obtain the second formula.

\end{proof}

\begin{proof}[Proof of Theorem \ref{val3}]
By Lemma \ref{serieswt0} and the transformation formulas \eqref{trans}, we have 
\begin{align*}
L(f, 3) &= \frac{\pi^{2}}{128} \int_{0}^{1} 
\alpha^{2}(1-\alpha)^{1/2} 
{_{2}F_{1}}
\left[ \left.
\begin{matrix}
1, 1 \\
2
\end{matrix}
\right| \alpha
\right]
{_{2}F_{1}}
\left[ \left.
\begin{matrix}
\frac{1}{2}, \frac{1}{2} \\
1
\end{matrix}
\right| \alpha
\right]
\frac{d\alpha}{\alpha (1-\alpha)}, \\
L(g,3) &= \frac{\pi^{2}}{64}
\int_{0}^{1}\alpha^{3/2}(1-\alpha)^{1/2} 
{_{2}F_{1}}\left[ \left.
\begin{matrix}
\frac{1}{2}, 1 \\
\frac{3}{2}
\end{matrix}
\right| \alpha \right]
{_{2}F_{1}}\left[ \left.
\begin{matrix}
\frac{1}{2}, \frac{1}{2} \\
1
\end{matrix}
\right| \alpha \right] 
\frac{d\alpha}{\alpha (1-\alpha)}.
\end{align*}
If we use the series expansions of hypergeometric functions and integrate term-by-term, we obtain 
\begin{align*}
L(f, 3)
&=\frac{\pi^{2}}{128} \int_{0}^{1} 
\alpha^{2}(1-\alpha)^{1/2} 
{_{2}F_{1}}
\left[ \left.
\begin{matrix}
1, 1 \\
2
\end{matrix}
\right| \alpha
\right]
{_{2}F_{1}}
\left[ \left.
\begin{matrix}
\frac{1}{2}, \frac{1}{2} \\
1
\end{matrix}
\right| \alpha
\right]
\frac{d\alpha}{\alpha (1-\alpha)} \\
&= \frac{\pi^{2}}{128}\sum_{m,n=0}^{\infty} \frac{(1)_{m}^{2}}{(2)_{m}(1)_{m}} \frac{\left( \frac{1}{2} \right)_{n}^{2}}{(1)_{n}^{2}} 
\int_{0}^{1} \alpha^{2+m+n}(1-\alpha)^{\frac{1}{2}} \frac{d\alpha}{\alpha (1-\alpha)} \\
&= \frac{\pi^{2}}{128}\sum_{m,n=0}^{\infty} \frac{(1)_{m}^{2}}{(2)_{m}(1)_{m}} \frac{\left( \frac{1}{2} \right)_{n}^{2}}{(1)_{n}^{2}} 
\frac{\Gamma\left( 2+m+n \right)\Gamma\left( \frac{1}{2} \right)}{\Gamma\left( \frac{5}{2}+m+n \right)}\\
&= \frac{\pi^{2}}{96} F_{1;1;1}^{1;2;2}\left[\left.
\begin{matrix}
2 \\
\frac{5}{2}
\end{matrix}
; 
\begin{matrix}
1, 1 \\
2
\end{matrix}
;
\begin{matrix}
\frac{1}{2}, \frac{1}{2} \\
1
\end{matrix}
\right| 1, 1
\right]. 
\end{align*}

By similar computations, we have the hypergeometric expression of the value $L(g, 3)$.

\end{proof}

\section{Proof of Theorem \ref{val4}}\label{proof2}
Similarly to the computations in the proof of Proposition \ref{intrep3}, 
we obtain the following integral expressions of the $L$-values $L(f,4)$ and $L(g,4)$. 

\begin{prop}\label{intrep4}
\begin{enumerate}
\item 
\begin{align*}
L(f, 4) = \frac{\pi^{3}}{48} \int_{0}^{1}
(2\theta_{4}^{8}(q^{2}) - \theta_{4}^{8}(q))
\left( \sum_{n,r=1}^{\infty} \frac{\chi_{-4}(n)}{(2r-1)^{2}} q^{n(r-1/2)} \right) \frac{dq}{q}.
\end{align*}

\item 
\begin{align*}
L(g,4) = 
\frac{\pi^{3}}{48} \int_{0}^{1}
(2\theta_{4}^{8}(q^{4}) - \theta_{4}^{8}(q^{2}))
\left( \sum_{n,r=1}^{\infty} \frac{\chi_{-4}(n)}{(2r-1)^{2}} q^{n(r-1/2)} \right) \frac{dq}{q}.
\end{align*}

\end{enumerate}
\end{prop}
\begin{proof}
We prove the formula for $L(f, 4)$ only. The formula for $L(g,4)$ is obtained by similar computations. 

By \eqref{mellin}, we have
\begin{align*}
L(f, 4) = \frac{1}{6}\int_{0}^{1} \frac{1}{16}\theta_{2}^{4}(q)\theta_{4}^{2}(q) (\log q)^{3} \frac{dq}{q} 
=\frac{\pi^{4}}{96}\int_{0}^{\infty} \theta_{2}^{4}(e^{-\pi u})\theta_{4}^{2}(e^{-\pi u}) u^{3} du .
\end{align*}
By \eqref{inv}, \eqref{lam1}, \eqref{lam2} and 
the variable transformation $u \mapsto (k-1/2)u/(2r-1)$, we have 
\begin{align*}
L(f,4) =\frac{\pi^{4}}{12}\int_{0}^{\infty}
\left( \sum_{s,k=1}^{\infty} (2k-1)^{3}e^{\frac{-\pi u (2s-1)(2k-1)}{2}} \right) 
\left( \sum_{n,r=1}^{\infty} \frac{\chi_{-4}(n)}{(2r-1)^{2}} e^{\frac{-\pi n(2r-1)}{u}} \right) u^{2}du .
\end{align*}

The first series in the integral is a theta product. 

\begin{lemm}
\begin{align*}
\sum_{s,k=1}^{\infty} (2k-1)^{3}q^{(2s-1)(2k-1)}
= \frac{\theta_{2}^{8}(q^{1/2}) - 8\theta_{2}^{8}(q) }{256}.
\end{align*}
\end{lemm}
\begin{proof}
We have
\begin{align*}
\sum_{s,k=1}^{\infty} (2k-1)^{3}q^{(2s-1)(2k-1)}
&=\sum_{s,k=1}^{\infty} (k^{3}q^{sk} -9k^{3}q^{2sk} +8k^{3}q^{4sk}) \\
&= \frac{M(q)-9M(q^{2}) + 8M(q^{4}) }{240}, 
\end{align*}
where 
\begin{align*}
M(q) := 1+ 240\sum_{s,k=1}^{\infty}k^{3}q^{sk} .
\end{align*}
We know that $\theta_{2}(q)$ has the connection with $M(q)$ \cite[p.207, Theorem 3.39 (3.101)]{cooper}
\begin{align*}
\theta_{2}^{8}(q^{\frac{1}{2}}) = \frac{16}{15} (M(q) - M(q^{2})). 
\end{align*}
By this identity, we obtain the lemma. 

\end{proof}

If we use the lemma above and \eqref{inv}, we have 
\begin{align*}
L(f, 4) &= \frac{\pi^{4}}{3072} \int_{0}^{\infty}
(\theta_{2}^{8}(e^{-\pi u / 4}) - 8\theta_{2}^{8}(e^{-\pi u /2}))
\left( \sum_{n,r=1}^{\infty} \frac{\chi_{-4}(n)}{(2r-1)^{2}} e^{\frac{-\pi n(2r-1)}{u}} \right) u^{2}du \\
&=\frac{\pi^{4}}{24} \int_{0}^{\infty}
(2\theta_{4}^{8}(e^{-4 \pi / u }) - \theta_{4}^{8}(e^{- 2\pi / u }))
\left( \sum_{n,r=1}^{\infty} \frac{\chi_{-4}(n)}{(2r-1)^{2}} e^{\frac{-\pi n(2r-1)}{u}} \right) \frac{du}{u^{2}}.  
\end{align*}
By changing the variables $u \mapsto 1/u$ and $q = e^{-2\pi u}$, we obtain the proposition.

\end{proof}

We remark that the series in the integrals in Proposition \ref{intrep4} can be expressed in terms of 
generalized hypergeometric functions  \cite[(2.2)]{duke} 
\begin{align}
\sum_{n,r=1}^{\infty} \frac{\chi_{-4}(n)}{(2r-1)^{2}} q^{n(r-1/2)}
= \frac{\alpha^{1/2}}{4}
\frac{
{_{3}F_{2}}\left[ \left. 
\begin{matrix}
1,1,1 \\
\frac{3}{2}, \frac{3}{2}
\end{matrix}
\right| \alpha
\right]
}{{_{2}F_{1}}\left[ \left. 
\begin{matrix}
\frac{1}{2}, \frac{1}{2} \\
1
\end{matrix}
\right| \alpha
\right]} . \label{ram}
\end{align}

\begin{proof}[Proof of Theorem \ref{val4}]
By \eqref{trans} and \eqref{ram}, we obtain 
\begin{align*}
L(f,4) 
&= \frac{\pi^{3}}{192}
\int_{0}^{1} (\alpha^{1/2} + \alpha^{3/2}) (1-\alpha)
{_{3}F_{2}}\left[ \left. 
\begin{matrix}
1,1,1 \\
\frac{3}{2}, \frac{3}{2}
\end{matrix}
\right| \alpha
\right]
{_{2}F_{1}}\left[ \left. 
\begin{matrix}
\frac{1}{2}, \frac{1}{2} \\
1
\end{matrix}
\right| \alpha
\right] 
\frac{d\alpha}{\alpha (1-\alpha)} , \\
L(g,4) 
&=  \frac{\pi^{3}}{384}
\int_{0}^{1} ((1-\alpha)^{1/2} + (1-\alpha)^{3/2}) \alpha^{1/2}
{_{3}F_{2}}\left[ \left. 
\begin{matrix}
1,1,1 \\
\frac{3}{2}, \frac{3}{2}
\end{matrix}
\right| \alpha
\right]
{_{2}F_{1}}\left[ \left. 
\begin{matrix}
\frac{1}{2}, \frac{1}{2} \\
1
\end{matrix}
\right| \alpha
\right] 
\frac{d\alpha}{\alpha (1-\alpha)} .
\end{align*}
Here we used
\begin{align*}
2\theta_{4}^{8}(q^{2}) - \theta_{4}^{8}(q) &= (1+\alpha)(1-\alpha)\theta_{3}^{8}(q), \\
2\theta_{4}^{8}(q^{4}) - \theta_{4}^{8}(q^{2}) &= \frac{1}{2} ((1-\alpha)^{1/2} + (1-\alpha)^{3/2})\theta_{3}^{8}(q),
\end{align*}
which follow from the formulas \cite[p. 34, (2.1.7i), (2.1.7ii)]{borweins}
\begin{align*}
2\theta_{3}^{2}(q^{2}) = \theta_{3}^{2}(q) + \theta_{4}^{2}(q), \hspace{3mm} 
\theta_{3}(q)\theta_{4}(q)=\theta_{4}^{2}(q^{2}). 
\end{align*}
Then the hypergeometric expressions can be proved by interchanging of the order of summation and integration.

\end{proof}

\section*{Acknowledgment}
The author expresses his gratitude to Noriyuki Otsubo for a lot of helpful comments on a draft version of this paper.

\end{document}